\documentclass[12pt,twoside, leqno]{article}

\usepackage{amssymb,amsmath,amsfonts,amsthm,color,mathrsfs}
\usepackage[Symbol]{upgreek}
\usepackage{txfonts}
\usepackage[nottoc,notlot,notlof]{tocbibind}
\usepackage[active]{srcltx}
\usepackage{hyperref}

\allowdisplaybreaks
\def\rr{{\mathbb R}}
\def\rn{{{\rr}^n}}

\def\cn{{\mathbb N}}

\def\supp{{\mathop\mathrm{\,supp\,}}}

\def\loc{{\mathop\mathrm{\,loc\,}}}

\def\BMO{{\mathrm{\,BMO\,}}}
\def\R{{\mathop\mathbb{R}}}

\def\la{\langle}
\def\ra{\rangle}

\newtheorem{thm}{Theorem}[section]
\newtheorem{lem}[thm]{Lemma}
\newtheorem{prop}[thm]{Proposition}

\numberwithin{equation}{section}

\begin{document}
\arraycolsep=1pt
\author{Renjin Jiang, Kangwei Li \& Jie Xiao}
\title{{\bf \Large Flow with $A_\infty(\mathbb R)$ density and transport equation in $\mathrm{BMO}(\mathbb R)$}
 \footnotetext{\hspace{-0.35cm} 2010 {\it Mathematics
Subject Classification}. Primary 34A30; Secondary 42A99; 34A12;
\endgraf{
{\it Key words and phrases: flow; $A_\infty(\mathbb R)$; $\mathrm{BMO}(\mathbb R)$; Zygmund condition}
\endgraf}}
\date{}}
\maketitle

\begin{center}
\begin{minipage}{11.5cm}\small
{\noindent{\bf Abstract}. We show that, if $b\in L^1(0,T;L^1_{\mathrm{loc}}(\mathbb{R}))$ has 
spatial derivative in the John-Nirenberg space $\mathrm{BMO}(\mathbb{R})$, then it generalizes a unique flow $\phi(t,\cdot)$ which has an $A_\infty(\mathbb R)$ density for each time $t\in [0,T]$.
Our condition on the map $b$ is optimal and we also get a sharp quantitative estimate for the density.
As a natural application we establish a well-posedness for the Cauchy problem of the transport equation in $\mathrm{BMO}(\mathbb R)$.
}\end{minipage}
\end{center}
\vspace{0.2cm}

\section{\large Statement of main results}\label{s1}

Given an integer $n\ge 1$, a real $T\ge t>0$ and an evolutionary self-map $b(t,\cdot)$ of $\mathbb R^n$ with $$b\in L^1(0,T;L^1_\loc(\R^n)),$$ consider the flow
$$
\phi(t,x)= x+\int_0^t b(r,\phi(r,x))\,dr.
$$
We are motivated by the composition and transportation problems in BMO space to answer the question:
\medskip

\boxed{\text{What condition is needed on a vector field such that it generalizes a flow $\phi$ with $A_\infty$ density?}}

\medskip
On $\R^n$, $n\ge 2$, the question has a satisfactory solution by Reimann \cite{R2} via the following $(Q)$-condition
$$
\sup_{(x,y,z)\in\R^n\times\R^n\times\R^n,\ |y|=|z|>0}
\left|\frac{\la y,b(x+y)-b(x)\ra}{|y|^2}-\frac{\la z,b(x+z)-b(x)\ra}{|z|^2}\right|<\infty\leqno(Q)
$$
which is equivalent to the anti-conformal part
$$
S_Ab=\dfrac{1}{2}(Db + Db^{T})-\dfrac{\mathrm{div}\,b}{n}\, I_{n\times n}$$
is bounded - moreover (cf. \cite{R2}) -
$$
S_Ab\in L^\infty(\R^n)\Rightarrow Db\in\BMO(\R^n).
$$
More precisely, \cite{R2} shows if $b$ satisfies $(Q)$ then it generalizes a unique flow $\phi(t,x)$,
which at each time $t$ is a quasi-conformal mapping and so the Jacobian $J_\phi$ of $\phi$ is of $A_\infty(\R^n)$ (cf. \cite{BHS}) where
 $$0\le w\in A_\infty(\mathbb R^n)\Leftrightarrow
 [w]_{A_\infty(\R^n)}=\sup_{\text{cubes}\, I\subset\R^n} \left(\frac{1}{|I|}\int_I w\,dx\right)\exp\left(-\frac {1}{|I|}\int_I\log w\,dx\right)<\infty.
 $$
 However, less known is the situation on $\mathbb R$. Note that the $1$-dimensional $(Q)$-condition coincides with the Zygmund condition for a constant $C>0$:
$$
|b(x+y)+b(x-y)-2b(x)|\le C|y|\ \ \forall\ \ (x,y)\in\mathbb R\times\mathbb R.\leqno(Z)
$$
Reimann \cite{R2} showed that for functions satisfying $(Q)$ the induced flows are
quasi-symmetric mappings - unfortunately - quasi-symmetric mappings are not necessarily  absolutely continuous in $\R$ and a function satisfying $(Z)$ needs not be absolutely continuous (cf. \cite{BA56,R2} and \cite{DLN}). In view of this, some more restrictions on $b$ seem to be necessary for the generalized flow to have an $A_\infty$ density. Observe that the notion $S_Ab=0$ in $\mathbb R$ does not carry any information.

In this paper, we show that if $b'$ is of $\BMO(\mathbb R)$ then $b$ generalizes a (unique) flow with $A_\infty(\R)$ densities. To see this clearly, recall that
$$
f\in \BMO(\mathbb R^n)\Leftrightarrow \|f\|_{\BMO(\mathbb R^n)}=\sup_{\text{cubes}\, I\subset\R^n}|I|^{-1}\int_{I}|f-f_I|\,dx<\infty,
$$
where
$$
f_I=|I|^{-1}{\int_I f(x)\,dx}
$$ denotes the integral average of $f$ over $I$ whose Lebesgue measure is written as $|I|$.
Since all constant functions have zero $\BMO(\R^n)$-norm, and any constant does effect the flow,
we make a modification on $\BMO(\R^n)$ functions $f$ as
$$\|f\|_\ast=\|f\|_{\BMO(\R^n)}+\int_{B(0,1)}|f|\,dx,$$
where $B(0,1)$ is the unit ball of $\R^n$.
Obviously,  $$f\in \BMO(\R^n)\Leftrightarrow\|f\|_\ast<\infty,$$
however, $\|f\|_\ast$ is not comparable to $\|f\|_{\BMO(\R^n)}$.
In what follows,
$$\frac{\partial}{\partial x}b(t,x)\in L^1(0,T;\mathrm{BMO}(\mathbb R))
$$
stands for
$$\int_0^T\Bigg\|\frac{\partial}{\partial x}b(t,x)\Bigg\|_\ast\,dt<\infty.$$


Our first main result reads as follows.
\begin{thm}\label{main}
Let \begin{equation}\label{e11}b(t,x):\,[0,T]\times\rr\mapsto \rr\ \text{be in}\ L^1(0,T;L^{1}_\loc(\R))\ \ \text{with}\ \
\frac{\,\partial b(t,x)}{\,\partial x}\in L^1(0,T;\mathrm{BMO}(\mathbb R)).
\end{equation}
Then there exists a unique flow $\phi(t,x)$ satisfying
$$
\begin{cases}
\dfrac{\partial}{\partial t}\, \phi(t, x)=b(t,\phi(t,x))\ &\forall\ (t,x)\in [0,T]\times\mathbb R;\\
\phi(0,x)=x\ \ &\forall\ x\in\mathbb R.
\end{cases}
$$
Moreover, for each $t\in [0,T]$,
$$
\left|\dfrac{\partial}{\partial x}\phi(t,\cdot)\right|
$$
is an $A_\infty(\R)$-weight, and there exist constants $C_1,c>0$ such that
\begin{equation}\label{bmo-flow-1}
\left\| \log\Big|\dfrac{\partial}{\partial x}\phi(t,\cdot)\Big|\right\|_{\BMO(\mathbb R)}
\le \frac{\int_{0}^t C_1\left\|\dfrac{\partial}{\partial x}b(s,\cdot)\right\|_{\BMO(\mathbb R)} \,ds}{ \exp\left(-c\int_{0}^t\left\|\dfrac{\partial}{\partial x}b(s,\cdot)\right\|_{\BMO(\mathbb R)}\,ds\right)}.
\end{equation}
\end{thm}
Some remarks are in order. First, from the well-known fact that the logarithm of an $A_\infty$
weight is a BMO function (see Lemma \ref{weight-bmo}) and the formula
$$\log\Big|\dfrac{\partial}{\partial x}\phi(t,x)\Big|=\int_0^t\dfrac{\partial}{\partial x}b(s,\phi(s,x))\,ds\in  \BMO(\R),$$
we see that our condition \eqref{e11} is critical, i.e., for each $t$,
$$
x\mapsto\frac{\partial}{\partial x}b(t,x)
$$ is necessarily a $ \BMO(\R)$-function.
Second, taking
$$b(x)=x\log|x|$$
for example, indicates that $b$ generalizes a flow $\phi(t,x)$ with
$$
\begin{cases}
\phi(t,x)=\mathrm{sign} x  \,|x|^{e^{t}}\\
\dfrac{\partial}{\partial x}\phi(t,x)=e^t|x|^{e^t-1}\in A_\infty(\R)\\
\left\|\log\Big|\dfrac{\partial}{\partial x}\phi(t,x)\Big|\right\|_{ \BMO(\R)}\le (e^t-1)\|\log|x|\|_{ \BMO(\R)}\le Cte^t.
\end{cases}
$$
This implies that our estimate \eqref{bmo-flow-1} is sharp.

For the proof, we shall first provide a version of the result in smooth setting, namely,
\begin{equation}\label{smooth-bmo}
b\in L^1(0,T;C^1(\R))\ \ \text{with}\ \ \frac{\,\partial b(t,x)}{\,\partial x}\in L^1(0,T; \BMO(\R)),
\end{equation}
and then use the compactness argument based on development of non-smooth flows from
\cite{Am,CJMO1,CDL08,DPL}. Note that since the Zygmund condition is satisfied for $b$, existence and uniqueness follow
already from Reimann \cite{R2}. The key of the proof is to establish \eqref{bmo-flow-1}, which even in
the smooth setting seems non-trivial.
By the composition result of Jones \cite{J}, a homeomorphism $\phi$ preserves $\BMO(\R)$ if and only if
$\phi'$ is an $A_\infty(\R)$ weight. However, even we assume that $b$ is smooth on $\R$,
it seems mysteries to us whether one can prove the generalized flow carries $A_\infty(\R)$ density
directly from \eqref{e11}.

In order to overcome the difficulties,  we further consider the
simpler case
\begin{equation}\label{lipschitz-case}
b\in L^1(0,T;C^1(\R))\ \ \text{with}\ \ \frac{\,\partial b(t,x)}{\,\partial x}\in L^1(0,T;L^\infty(\R)),
\end{equation}
where the generalized flow carries $A_\infty(\R)$-density following from the Cauchy-Lipschitz theory.
Then we observe
that for a function $v$ with small $\BMO(\mathbb R)$-norm, $e^v$ lies in the $A_\infty(\mathbb R)$ class with its norm
controlled by the $\BMO(\mathbb R)$-norm of $v$ linearly. Then by using the flow with $A_\infty(\mathbb R)$-density in the smooth setting, a quantitative estimate of the norm of composition in $\BMO(\mathbb R)$,  and a bootstrap argument,
we succeed in showing \eqref{bmo-flow-1} in the Lipschitz case \eqref{lipschitz-case}.
Finally a truncation argument involving the Arzel\'a-Ascoli theorem allows us to pass to the case \eqref{smooth-bmo};
see Section 3.

One may wonder if a quantitative estimate of the $A_\infty(\mathbb R)$-norm of $$\left|\dfrac{\partial}{\partial x}\phi(t,\cdot)\right|$$
can be established. Although we do not know a positive answer, we doubt it since a quantitative bound for an $A_\infty(\mathbb R)$-weight $e^v$ holds only for $v$ with small $\BMO(\mathbb R)$-norm; see Lemma \ref{A2-bmo} and Lemma \ref{weight-bmo} below. However, there is a nice result regarding homeomorphisms preserving $A_p(\mathbb R)$-weights by \cite{JN}.

We next apply the result on flow to study the transportation problem in BMO space.
Besides its own interest, this problem and its dual equation also arise naturally from the study of conservation laws (see \cite{BJ99} for instance).
In \cite{CJMO0} (somewhat related to \cite{Muc}), a well-posedness of the Cauchy problem of the transport equation in $\BMO(\R^n)$ has been established for $n\ge 2$ and then pushed to the case $n=1$ in \cite{Xiao}. The main step over there is to use the hypothesis that
$$
(t,x)\mapsto \begin{cases} S_A b(t,x)\ &\ \forall\ \ n\ge 2\\
\frac{\partial}{\partial x}b(t,x)\ &\ \forall\ \ n=1
\end{cases}
\ \ \text{
belongs to}\ \ L^1(0,T; L^\infty(\mathbb R^n))\ \ \text{with a suitably small norm},
$$ the quasi-conformal flows of \cite{R2} and the composition results obtained in \cite{kyz,R1} for $n\ge 2$ (cf. \cite{KKSS,Vo,YYZ}) and in \cite{J} for $n=1$. But nevertheless, as our second main result we utilize Theorem \eqref{main} and \cite[Theorem]{J} to discover the following stronger well-posedness of the transport equation in $\BMO(\mathbb R)$.

\begin{thm}\label{bmo-transport}
Let $b(t,x):\,[0,T]\times\rr\mapsto \rr$ be in $L^1(0,T; L^{1}_\loc(\R))$ and satisfy
$$\frac{\,\partial b(t,x)}{\,\partial x}\in L^1(0,T;\mathrm{BMO}(\mathbb R)).$$
Then for $u_0\in \BMO(\R)$ there exists a unique solution $u\in L^\infty(0,T;\BMO(\mathbb R))$ to the Cauchy problem of the transport equation
$$
\begin{cases}
    \Bigg(\dfrac{\partial u}{\partial t}-b\cdot \dfrac{\partial u}{\partial x}\Bigg)(t,x)=0\ & \forall\ (t,x)\in (0,T)\times\R;\\
      u(0,x)=u_0(x)\ & \ \forall\ x\in\R.
      \end{cases}
      $$
Moreover, for each $t\in [0,T]$, it holds that
$$
\begin{cases}
u(t,x)=u_0(\phi(t,x));\\
\frac{\,\partial}{\,\partial t}\phi(t,x)=b(t,\phi(t,x)),
\end{cases}
$$
and there exist $C_2,c>0$ such that
\begin{equation}\label{est-solution}
\|u\|_{\BMO(\mathbb R)}\le C_2\|u_0\|_{\BMO(\mathbb R)} \exp\left(c\int_{0}^t\left\|\dfrac{\partial}{\partial x}b(s,\cdot)\right\|_{\BMO(\mathbb R)}\,ds\right).
\end{equation}
\end{thm}
Based on the duality of Hardy space $H^1$ and BMO by Fefferman and Stein \cite{FS72}, the above theorem provides
the existence of solutions in Hardy space $H^1$ to the continuity equation
$$
\begin{cases}
    \Bigg(\dfrac{\partial u}{\partial t}-\dfrac{\partial }{\partial x}(bu)\Bigg)(t,x)=0\ & \forall\ (t,x)\in (0,T)\times\R;\\
      u(0,x)=u_0(x)\ & \ \forall\ x\in\R;
      \end{cases}
      $$
see \cite{CJyMO} for a study of the equation in higher dimensions and a proof of uniqueness (cf. \cite[Theorem 3]{CJyMO}).

The paper is organized as follows. In Section \ref{s2}, we recall and establish some results concerning
Muckenhoupt weights, $ \BMO(\R)$, and continuity estimates. In Section \ref{s3}, we present the key a priori estimation for the flow, i.e., the version of Theorem \ref{main} in the smooth setting. In Section \ref{s4}, we verify the above main results.

\medskip

\noindent{\bf Notation}. {\it In the above and below, $C,C_1,C_2,...$ and $c,c_1,c_2,...$ stand for positive constants.}

\section{\large Weights and bounded mean oscillation}\label{s2}

For a locally integrable function $f$ and an open interval $I\subset\R$, we denote by $f_I$
the integral average of $f$ on $I$.
We say that a locally integrable nonnegative function $w$ belongs to the Muckenhoupt $A_p(\mathbb R)$ class,  $1<p<\infty$, if
$$[w]_{A_p(\mathbb R)}=\sup_{\text{intervals}\, I\subset \R} \left(\frac{1}{|I|}\int_I w\,dx\right) \left(\frac{1}{|I|}\int_I w^{\frac{1}{1-p}}\,dx\right)^{p-1}<\infty,$$
and that $w\in A_\infty(\mathbb R)$, if
$$[w]_{A_\infty(\mathbb R)}=\sup_{\text{intervals}\, I\subset \R} \left(\frac{1}{|I|}\int_I w\,dx\right)\exp\left(-\frac {1}{|I|}\int_I\big(\log w\big)\,dx\right)<\infty.$$

Note that, if $w>0$ a.e., then $[w]_{A_\infty(\mathbb R)}\ge 1$ follows from the Jensen inequality that
\begin{align*}
[w]_{A_\infty(\mathbb R)}\ge w_I\exp\left([-\log w]_I\right)\ge \exp\left((\log w)_I\right)\exp\left([-\log w]_I\right)=1,
\end{align*}
and similarly
$$[w]_{A_p(\mathbb R)}\ge [w]_{A_\infty(\mathbb R)}\ \ \forall\ \ p\in (1,\infty).
$$

We need the following quantitative version of reverse H\"older inequality for $A_\infty(\mathbb R)$-weight from \cite{HP13}; see also \cite{LOP17}.
\begin{lem}\label{rhi-weight}
Let $w\in A_\infty(\mathbb R)$ and $I\subset\R$ be an arbitrary interval. Then there exits
$$
\begin{cases}
\tau>0;\\
r_w=1+\Big({\tau[w]_{A_\infty(\mathbb R)}}\Big)^{-1};\\
\epsilon_w=\Big({1+\tau[w]_{A_\infty(\mathbb R)}}\Big)^{-1},
\end{cases}
$$
such that
$$
\begin{cases}
\left(|I|^{-1}\int_I w^{r_w}\,dx\right)^{1/r_w}\le 2|I|^{-1}\int_I w\,dx;\\
\frac{w(E)}{w(I)}=\frac{\int_E w(x)\,dx}{\int_I w(x)\,dx}\le 2\left(\frac{|E|}{|I|}\right)^{\epsilon_w}\ \ \text{for any measurable set}\ E\subset I.
\end{cases}
$$
\end{lem}

 By \cite[Theorem]{J}, we know that an increasing homeomorphism $\varphi$ of  $\R$ preserves $\BMO$ if and only if $\varphi'$ belongs to $A_\infty(\mathbb R)$. By using the previous lemma we deduce the following quantitative version; see \cite{ACS} for an explicit bound in terms of reverse H\"older index and \cite{Bl,FHS,Fo} for related results.
\begin{lem}\label{bmo-comp} Let $\varphi$ be an increasing homeomorphism on $\rr$ with $\varphi'\in A_\infty(\mathbb R)$. Then there is $C_3>0$ such that
$$\|f\circ \varphi^{-1}\|_{ \BMO(\R)}\le C_3[\varphi']_{A_\infty(\mathbb R)}\|f\|_{ \BMO(\R)}.$$
\end{lem}
\begin{proof}
Recall that for a $\BMO(\R)$-function $f$, the John-Nirenberg inequality states that, for all $I\subset \rr$,
 there exists $c_1,c_2>0$ such that
$$\big|\{x\in I:\,|f(x)-f_I|>\lambda\}\big|\le c_1|I|\exp\left(-\frac{c_2\lambda}{\|f\|_{\BMO(\mathbb R)}}\right)\  \ \forall\ \ \lambda>0;$$
see \cite{Gra04} for instance.

Suppose that $\varphi$ is an increasing homeomorphism of $\rr$ with $\varphi'\in A_\infty(\mathbb R)$. By \cite[Theorem]{J},
we have
$$
f\circ \varphi^{-1}\in \BMO.
$$

For every interval $$I=(a,b)\subset \rr,$$ set
$$E_\lambda=\Big\{x\in I:\,|f\circ\varphi^{-1}(x)-f_{\varphi^{-1}(I)}|>\lambda\Big\}.$$
Then
$$\varphi^{-1}(E_\lambda)=\Big\{y\in \varphi^{-1}(I):\,|f(y)-f_{\varphi^{-1}(I)}|>\lambda\Big\},$$
and hence, by Lemma \ref{rhi-weight} and the John-Nirenberg inequality, we get
$$\frac{|E_\lambda|}{|I|}\le 2\left(\frac{|\varphi^{-1}(E_\lambda)|}{\varphi^{-1}(I)}\right)^{\epsilon_w}
\le 2c_1\exp\left(-\frac{c_2\epsilon_w\lambda}{\|f\|_{\BMO(\mathbb R)}}\right)\ \ \text{where}\ \
\epsilon_w=\big(1+\tau [\varphi']_{A_\infty(\mathbb R)}\big)^{-1},
$$
thereby finding
$$\|f\circ\varphi^{-1}\|_{\BMO(\mathbb R)} \le C(1+\tau [\varphi']_{A_\infty(\mathbb R)})\|f\|_{\BMO(\mathbb R)}\le C_3[\varphi']_{A_\infty(\mathbb R)}\|f\|_{\BMO(\mathbb R)},$$
where we have used the fact that $\varphi$ is an increasing homeomorphism on $\rr$ with $$[\varphi']_{A_\infty(\mathbb R)}\ge 1.$$
\end{proof}

The following result is well-known; see \cite{CPP10,Gra04} for instance.
\begin{lem}\label{A2-bmo} There exists  $\alpha<1<\beta$ such that for
	$$
	\begin{cases}
	f\in \BMO(\R);\\
	s\in\rr;\\
|s|\le {\alpha}{\|f\|^{-1}_{ \BMO(\R)}},
\end{cases}
$$
it holds that
$$
e^{sf}\in A_2(\mathbb R)\ \ \text{with}\ \ [e^{sf}]_{A_2(\mathbb R)}\le \beta^2.
$$
\end{lem}
Here it is perhaps appropriate to mention that the requirement
$$|s|\le {\alpha}{\|f\|^{-1}_{ \BMO(\R)}}$$ is critical since
$$x\mapsto f(x)=\log|x|$$ is in $ \BMO(\R)$ but $$x\mapsto e^{-f(x)}=|x|^{-1}$$ is not a Muckenhoupt weight.

\begin{lem}\label{weight-bmo}
If $$0\le w\in A_\infty(\mathbb R)$$ then
$$
\|\log w\|_{\BMO(\mathbb R)}\le 2\log ([w]_{A_\infty(\mathbb R)}+1).
$$
Conversely, if $$v\in \BMO(\mathbb R)\ \ \&\ \ v\ge 0\ \ \text{a.e. on}\ \ \mathbb R$$
then there exists a sufficiently small $\epsilon_0\in (0,1]$ such that  $$\|v\|_{\BMO(\mathbb R)}<\epsilon_0\Rightarrow
e^v\in A_\infty(\mathbb R)\ \ \text{with}\ \
[e^v]_{A_\infty(\mathbb R)}\le 1+C_4\|v\|_{ \BMO(\R)}.$$
\end{lem}
\begin{proof}
	On the one hand, for any $0\le w\in A_\infty(\mathbb R)$ we have
	\begin{align*}
		\int_I \left|\log w-(\log w)_I\right|\,dx&=\int_I \left[\log w-(\log w)_I\right]_+\,dx+\int_I \left[\log w-(\log w)_I\right]_-\,dx\\
		&=2 \int_I \left[\log w-(\log w)_I\right]_+\,dx,
	\end{align*}
	where $[f]_+$ and $[f]_-$ denotes the positive and negative parts of $f$ respectively.
	In virtue of Jensen's inequality we obtain
	\begin{align*}
		|I|^{-1}\int_I \left|\log w-(\log w)_I\right|\,dx
		&=2 |I|^{-1}\int_I \left[\log w-(\log w)_I\right]_+\,dx\\
		&\le 2\log\left(|I|^{-1} \int_I \exp\left[\log w-(\log w)_I\right]_+\,dx\right) \\
		&\le 2\log \left(|I|^{-1}\int_I \exp\left[\log w-(\log w)_I\right]\,dx+1\right)\\
		&\le 2\log ([w]_{A_\infty(\mathbb R)}+1),
	\end{align*}
	whence
	$$
	\|\log w\|_{\BMO(\mathbb R)}\le 2\log ([w]_{A_\infty(\mathbb R)}+1).
	$$
	
	On the other hand, note that
\begin{align}\label{est-a}
[e^v]_{A_\infty(\mathbb R)}=\left(\sup_{I=(a,b)\subset \R}|I|^{-1}\int_I e^{v(x)}\,dx\right) \exp\left([-v]_I\right)=
\sup_{I=(a,b)\subset \R}|I|^{-1}\int_I e^{v(x)-v_I}\,dx.
\end{align}
So, if $v\in  \BMO(\R)$, then the John-Nirenberg inequality gives
$$|\{x\in I:\,|v(x)-v_I|>\lambda\}|\le c_1|I|\exp\left(-\frac{c_2\lambda}{\|v\|_{\BMO(\mathbb R)}}\right).$$
Inserting this into \eqref{est-a}, we find that if
$$\|v\|_{\BMO(\mathbb R)}<c_2$$ then
\begin{align*}
|I|^{-1}\int_I e^{v(x)-v_I}\,dx&=\frac{1}{|I|}\int_{x\in I:\, v(x)-v_I<0} e^{v(x)-v_I}\,dx+
\frac{1}{|I|}\int_{x\in I:\, v(x)-v_I\ge 0} e^{v(x)-v_I}\,dx \\
&\le 1+ c_1\int_{0}^\infty \exp\left(\lambda-\frac{c_2\lambda}{\|v\|_{\BMO(\mathbb R)}}\right)\,d\lambda \\
&\le 1+\frac{c_1\|v\|_{\BMO(\mathbb R)}}{c_2-\|v\|_{\BMO(\mathbb R)}}.
\end{align*}
Accordingly,
$$\|v\|_{\BMO(\mathbb R)}<2^{-1}c_2\Rightarrow
[e^v]_{A_\infty(\mathbb R)}\le 1+{2c_1c_2^{-1}\|v\|_{\BMO(\mathbb R)}}.$$
Letting $$\epsilon_0=\min\{1,2^{-1}c_2\}$$ yields the assertion.
\end{proof}

\begin{prop}\label{cont-est}
Suppose that $b\in L^1_\loc(\R)$ has its derivative $b'\in  \BMO(\R)$. Then $b$ satisfies the Zygmund
condition with
\begin{equation*}\label{modulus-cont-1}
{|b(x+y)+b(x-y)-2b(x)|}\le 2|y|\|b'\|_{ \BMO(\R)}\ \ \forall\ \ (x,y)\in\mathbb R\times\mathbb R.
\end{equation*}
\end{prop}
\begin{proof} This follows from
\begin{align*}
&|b(x+y)+b(x-y)-2b(x)|\\
&\quad=\left|\int_x^{x+y} b'(z)\,dz-\int_{x-y}^xb'(z)\,dz\right|\\
&\quad\le \left|\int_x^{x+y} b'(z)\,dz-\frac 12\int_{x-y}^{x+y} b'(z)\,dz\right|+\left|\frac 12\int_{x-y}^{x+y} b'(z)\,dz-\int_{x-y}^xb'(z)\,dz\right|\\
&\quad \le \int_x^{x+y} \left|b'(z)-b'_{[x-y,x+y]}\right|\,dz+\int_{x-y}^{x} \left|b'(z)-b'_{[x-y,x+y]}\right|\,dz\\
&\quad\le \int_{x-y}^{x+y} \left|b'(z)-b'_{[x-y,x+y]}\right|\,dz\\
&\quad \le 2|y|\|b'\|_{ \BMO(\R)}.
\end{align*}
\end{proof}

Recall that for a $ \BMO(\R)$ function $f$ we have
$$
\|f\|_\ast=\|f\|_{ \BMO(\R)}+\int_{[-1,1]}|f|\,dx<\infty.
$$
In what follows, for a positive constant $C$, denote by $$\log^+ C=\max\{1,\,\log C\}.$$
\begin{prop}\label{cont-est-2}
Suppose that $b\in L^1_\loc(\R)$ has its derivative $b'\in  \BMO(\R)$. Then $b$ satisfies
$$|b(x)-b(0)|\le C_5\|b'\|_\ast|x|(1+|\log|x||)\ \ \forall\ \ x\in\R$$
and
$$|b(x+h)-b(x)|\le C_5\|b'\|_\ast(\log^+|x|)\left( |h| (1+|\log|h||)\right)\ \ \forall\ \ (x,h)\in\R\times \R.$$
\end{prop}
\begin{proof}
From \cite[Proposition 5]{R2} and Proposition \ref{cont-est} it follows that if
$$
y\neq 0;\ z\neq 0;\ x\in\R,
$$
then
\begin{equation}\label{zygmund}
\left|\frac{(y,b(x+y)-b(x))}{|y|^2}-\frac{(z,b(x+z)-b(x))}{|z|^2}\right|\le 5\|b'\|_{ \BMO(\R)}+\frac{\|b'\|_{ \BMO(\R)}}{\log 2}\left|\log\frac{|y|}{|z|}\right|.
\end{equation}
Letting $x=0$ and $z=1$ in \eqref{zygmund} gives the first inequality in Proposition \ref{cont-est-2} via
\begin{align*}
\left|b(y)-b(0)\right|&\le |y|\left(|b(1)-b(0)|+ 5\|b'\|_{ \BMO(\R)}+\frac{\|b'\|_{ \BMO(\R)}}{\log 2}\left|\log|y|\right|\right)\\
&\le C_5\|b'\|_\ast |y|(1+|\log|y|).
\end{align*}
Also, by using  structure of $\BMO(\R)$ (cf. \cite{Gra04}) we see that if $x\in\R$ then
\begin{eqnarray*}
|b(x+1)-b(x)|&&=\left|\int_x^{x+1}b'\,dy-\int_0^1b'\,dy\right|+\left|\int_0^1b'\,dy\right|\\
&&\le 2(\log^+|x|)\|b'\|_{ \BMO(\R)}+\left|\int_0^1b'\,dy\right|\\
&&\le 2(\log^+|x|)\|b'\|_\ast.
\end{eqnarray*}
This, along with \eqref{zygmund}, derives the second inequality in Proposition \ref{cont-est-2} via
\begin{eqnarray*}
\left|b(x+h)-b(x)\right|&&\le |h|\left(|b(x+1)-b(x)|+ 5\|b'\|_{ \BMO(\R)}+\frac{\|b'\|_{ \BMO(\R)}}{\log 2}\left|\log|h|\right|\right)\\
&&\le C_5\|b'\|_\ast(\log^+|x|) |h|(1+|\log|h|).
\end{eqnarray*}
\end{proof}

\section{\large Key a priori estimates for the flow}\label{s3}

We say that $\phi$ is a forward flow associated to $b$ if for each $s\in [0,T]$ and almost every $x\in\rn$  the map
$$t\mapsto\,|b(t, \phi_s(t, x))|\ \ \text{belongs to}\ \ L^1(s,T)
$$
and
$$\phi_s(t,x)= x+\int_s^t b(r,\phi_s(r,x))\,dr.$$
If the flow starts at $s=0$, then we simply denote $\phi_0(t,x)$ by $\phi(t,x)$.

Meanwhile, we say that $\tilde \phi$ is a backward flow associated to $b$ if for each $t\in [0,T]$ and almost every $x\in\rn$  the map
$$s\mapsto\,|b(s, \tilde \phi_t(s,x))|\ \ \text{belongs to}\ \ L^1(0,t)
$$ and
$$\tilde \phi_t(s, x) = x -\int_s^t b(r,\tilde \phi_t(r,x))\,dr.$$

\begin{thm}\label{pri-flow}
Let $$b(t,x):\,[0,T]\times\rr\mapsto \rr\ \ \text{be in}\ \ L^1(0,T;C^1(\rr))\ \ \text{with}\ \
\int_0^T\Bigg\|\frac{\,\partial b(t,\cdot)}{\,\partial x}\Bigg\|_{L^\infty(\rr)}\,dt<\infty.
$$
Then there exists a unique flow $\phi(t,x)$ satisfying
$$
\begin{cases}
\dfrac{\partial}{\partial t}\phi(t,x)=b(t,\phi(t,x))\ \ &\forall\ \ (t,x)\in [0,T]\times\R;\\
\phi_0(x)=x\ \ &\forall\ \ x\in\R.
\end{cases}
$$
Moreover, for each $t\in [0,T]$, it holds that
$$
\left\| \log\Big|\dfrac{\partial}{\partial x}\phi(t,x)\Big|\right\|_{\BMO(\mathbb R)}
\le \frac{\int_{0}^t C_6\left\|\dfrac{\partial}{\partial x}b(s,)\right\|_{\BMO(\mathbb R)} \,ds}{ \exp\left(-C_7\int_{0}^t\left\|\dfrac{\partial}{\partial x}b(s,)\right\|_{\BMO(\mathbb R)}\,ds\right)}.
$$
\end{thm}
\begin{proof} The argument is divided into four steps.
	\medskip
	
\noindent{\bf Step 1 - initialing argument}.
Since $$b(t,x):\,[0,T]\times\rr\mapsto \rr$$ satisfies $$b\in L^1(0,T;C^1(\rr))\ \ \text{with}\ \
\frac{\,\partial b(t,x)}{\,\partial x}\in L^1(0,T;L^\infty(\rr)),$$
the classical Cauchy-Lipschitz theory produces a unique flow $\phi_s(t,x)$ with
$$
\begin{cases}
\dfrac{\partial}{\partial t}\, \phi_s(t,x)=b(t,\phi_s(t, x))\ \ &\forall\ \ (t,x)\in [s,T]\times\R;\\
\phi_s(s,x)=x\ \ &\forall\ \ x\in\R.
\end{cases}
$$
Moreover, for each $t\in [s,T]$, $\phi_s(t,\cdot)$ is a bi-Lipschitz map on $\rr$.
Differentiating  the equation with respect to the spatial direction, we have
$$
\begin{cases}
\dfrac{\partial}{\partial x}\left(\dfrac{\partial}{\partial t}\, \phi_s(t,x)\right)=\left(\dfrac{\partial}{\partial x}b(t,\phi_s(t, x))\right)\dfrac{\partial}{\partial x}\, \phi_s(t,x);\\
\dfrac{\partial}{\partial t}\log\Big|\dfrac{\partial}{\partial x}\phi_s(t,x)\Big|=\dfrac{\partial}{\partial x}b(t,\phi_s(t, x)).
\end{cases}
$$

As $\phi_s(t,\cdot)$ is a bi-Lipschitz map on $\rr$ for each $t\in [s,T]$, its $x$-derivative has lower and upper bounds, i.e.,
$$e^{-\int_s^t A(r)\,dr}\le \left|\dfrac{\partial}{\partial x}\phi_s(t,x)\right|\le e^{\int_s^t A(r)\,dr},$$
where $$
A(r)=\Bigg\|\dfrac{\partial}{\partial x}b(r,\cdot)\Bigg\|_{L^\infty(\R)}.
$$
In particular, this implies that for each $t$, the function $$\bigg|\dfrac{\partial}{\partial x}\phi_s(t,\cdot)\bigg|$$ is an $A_\infty(\mathbb R)$-weight with
$$\left[\Big|\dfrac{\partial}{\partial x}\phi_s(t,\cdot)\Big|\right]_{A_\infty(\mathbb R)}\le e^{2\int_s^t A(r)\,dr}.$$
Note that the same estimate holds for the backward flow $\tilde \phi_t(s,x)$, which is the inverse of $\phi_s(t,x)$.

Upon applying Lemma \ref{bmo-comp}, we achieve
\begin{align}\label{est-log-bmo}
&\left\|\log\Big|\dfrac{\partial}{\partial x}\phi_s(t,\cdot)\Big|\right\|_{\BMO(\mathbb R)}\nonumber\\
&\ \ =\left\|\int_s^t\dfrac{\partial}{\partial x}b(r,\phi_s(r,\cdot))\,dr\right\|_{\BMO(\mathbb R)}\nonumber\\
&\ \ \le \int_s^t \left\|\dfrac{\partial}{\partial x}b(r,\phi_s(r,\cdot))\right\|_{\BMO(\mathbb R)}\,dr\\
&\ \ \le \int_s^t C_3\left\|\dfrac{\partial}{\partial x}b(r,\cdot)\right\|_{\BMO(\mathbb R)} \left[\dfrac{\partial}{\partial x} \tilde\phi_r(s,\cdot)\right]_{A_\infty(\mathbb R)}\,dr\nonumber\\
&\ \ \le \int_s^t C_3\left\|\dfrac{\partial}{\partial x}b(r,\cdot)\right\|_{\BMO(\mathbb R)} e^{2\int_s^r A(z)\,dz}\,dr.\nonumber
\end{align}

\noindent{\bf Step 2 - starting from short time}.
By letting $T_0>s\ge 0$ be small enough with
\begin{equation*}\label{time-lipschitz}
\int_s^{T_0} C_3\left\|\dfrac{\partial}{\partial x}b(r,\cdot)\right\|_{\BMO(\mathbb R)} e^{2\int_s^{r} A(z)\,dz}\,dr<\epsilon_0,
\end{equation*}
 where $\epsilon_0$ is as in Lemma \ref{weight-bmo}, we utilize \eqref{est-log-bmo} to get
\begin{align*}
\sup_{s\le t\le T_0}\left\{\left\|\log\Big|\dfrac{\partial}{\partial x}\phi_s(t,\cdot)\Big|\right\|_{\BMO(\mathbb R)},\,
\left\|\log\Big|\dfrac{\partial}{\partial x}\tilde\phi_t(s,\cdot)\Big|\right\|_{\BMO(\mathbb R)}\right\} <\epsilon_0.
\end{align*}
Hence, by applying Lemma \ref{weight-bmo}, we see
$$\left[\Big|\dfrac{\partial}{\partial x}\phi_s(t,\cdot)\Big|\right]_{A_\infty(\mathbb R)}<1+C_4\left\|\log\Big|\dfrac{\partial}{\partial x}\phi_s(t,\cdot)\Big|\right\|_{\BMO(\mathbb R)}.$$
Inserting this estimate into \eqref{est-log-bmo}, we conclude
\begin{align*}\label{est-log-bmo2}
\left\|\log\Big|\dfrac{\partial}{\partial x}\phi_s(t,\cdot)\Big|\right\|_{\BMO(\mathbb R)}&\le \int_s^t C_3\left\|\dfrac{\partial}{\partial x}b(r,\cdot)\right\|_{\BMO(\mathbb R)} \left[\Big|\dfrac{\partial}{\partial x} \tilde\phi_r(s, \cdot)\Big|\right]_{A_\infty(\mathbb R)}\,ds\nonumber\\
&\le \int_s^t C_3\left\|\dfrac{\partial}{\partial x}b(r,\cdot)\right\|_{\BMO(\mathbb R)} \left(1+C_4\left\|\log\Big|\dfrac{\partial}{\partial x}\tilde\phi_r(s,\cdot)\Big|\right\|_{\BMO(\mathbb R)}\right)\,ds.
\end{align*}
Set
$$
I_s(t)=\sup_{s\le r\le t}\left\{\left\|\log\Big|\dfrac{\partial}{\partial x}\phi_s(r,\cdot)\Big|\right\|_{\BMO(\mathbb R)},\,
\left\|\log\Big|\dfrac{\partial}{\partial x}\tilde\phi_r(s,\cdot)\Big|\right\|_{\BMO(\mathbb R)}\right\}.
$$
The above estimates yield
$$
I_s(t)\le \int_s^t C_3\left\|\dfrac{\partial}{\partial x}b(r,\cdot)\right\|_{\BMO(\mathbb R)} \left(1+C_4I_s(r)\right)\,dr\ \forall\ t\in [s,T_0].
$$
The Gronwall inequality then implies
\begin{equation}\label{est-log-bmo4}
I_s(t)\le \frac{\int_s^t C_3\left\|\dfrac{\partial}{\partial x}b(r,\cdot)\right\|_{\BMO(\mathbb R)}\,dr}{ \exp\left(-\int_s^t C_3C_4\left\|\dfrac{\partial}{\partial x}b(r,\cdot)\right\|_{\BMO(\mathbb R)}\,dr \right)}\ \forall\ t\in [s,T_0].
\end{equation}

\noindent{\bf Step 3 - removing the dependence of Lipschitz constant}.
Let $T_1\in (s,T]$ obey
\begin{equation}\label{time-bmo}
\int_s^{T_1} C_3\left\|\dfrac{\partial}{\partial x}b(r,\cdot)\right\|_{\BMO(\mathbb R)}\,dr \exp\left(C_3C_4\int_s^{T_1}\left\|\dfrac{\partial}{\partial x}b(r,\cdot)\right\|_{\BMO(\mathbb R)}\,dr\right)\le2^{-1}\epsilon_0,
\end{equation}
We claim that \eqref{est-log-bmo4} holds for all $t\in (s,T_1]$.

If $T_1\le T_0$, then the claim follows from \eqref{est-log-bmo4}.

Suppose now $T_0<T_1$. Assume that for some $t_0\in [T_0,T_1)$, \eqref{est-log-bmo4} holds for all $t\in (s,t_0]$.
Then
\begin{align*}\label{est-log-bmo6}
I_s(t_0)\le \frac{\int_s^{t_0} C_3\left\|\dfrac{\partial}{\partial x}b(r,\cdot)\right\|_{\BMO(\mathbb R)}\,dr}{ \exp\left(-\int_s^{t_0} C_3C_4\left\|\dfrac{\partial}{\partial x}b(r,\cdot)\right\|_{\BMO(\mathbb R)}\,dr \right)}\le2^{-1}\epsilon_0.
\end{align*}
Since
$$\frac{\,\partial b(t,\cdot)}{\,\partial x}\in L^1(0,T;L^\infty(\R)),$$
We can choose $ t_1\in (t_0,T_1]$ such that
\begin{equation}\label{extend-time1}
\int_{t_{0}}^{t_1} C_3\left\|\dfrac{\partial}{\partial x}b(r,\cdot)\right\|_{\BMO(\mathbb R)} e^{2\int_{t_{0}}^{r} A(z)\,dz}\,dr<\epsilon_0
\end{equation}
and
\begin{equation}\label{extend-time2}
\frac{\int_{t_0}^{t_1} C_3\left\|\dfrac{\partial}{\partial x}b(r,\cdot)\right\|_{\BMO(\mathbb R)}\,dr }{\exp\left(-\int_{t_0}^{t_1} C_3C_4\left\|\dfrac{\partial}{\partial x}b(r,\cdot)\right\|_{\BMO(\mathbb R)}\,dr \right)}<\frac{\epsilon_0}{2C_3(1+C_4 2^{-1}\epsilon_0)}.
\end{equation}

The same argument as in proving \eqref{est-log-bmo4} then implies that for $t_0<t\le t_1$ it holds
\begin{align}\label{est-log-bmo5}
I_{t_0}(t)\le \frac{\int_{t_0}^t C_3\left\|\dfrac{\partial}{\partial x}b(r,\cdot)\right\|_{\BMO(\mathbb R)}\,dr}{\exp\left(-\int_{t_0}^t C_3C_4\left\|\dfrac{\partial}{\partial x}b(r,\cdot)\right\|_{\BMO(\mathbb R)}\,dr \right)}<\frac{\epsilon_0}{2C_3(1+C_4 2^{-1}\epsilon_0)}.
\end{align}

For any $t\in (t_0,t_1]$,  we have via the semigroup property of the flow that
$$\phi_s(t,x)=\phi_{t_{0}}(t, \phi_s(t_0,x)).$$
By applying Lemma \ref{bmo-comp}, Lemma \ref{weight-bmo} and \eqref{est-log-bmo5}, we find
\begin{align*}\label{est-comp-bmo}
&\left\|\log\Big|\frac{\,\partial}{\,\partial x}\phi_s(t,\cdot)\Big|\right\|_{\BMO(\mathbb R)}\nonumber\\
&\ \ =\left\|\log\Big|\frac{\,\partial}{\,\partial x}\phi_{t_{0}}(t, \phi_s(t_{0},\cdot))\Big|\right\|_{\BMO(\mathbb R)}\nonumber\\
&\ \ \le \left\|\log\Big|\frac{\,\partial}{\,\partial z}\phi_{t_{0}}(t,z)\large|_{z= \phi_s(t_0,\cdot)}\Big|\right\|_{\BMO(\mathbb R)}+\left\|\log\Big|\frac{\,\partial}{\,\partial x}\phi_{s}(t_0,\cdot)\Big|\right\|_{\BMO(\mathbb R)}\\
&\ \ \le C_3\left\|\log\Big|\frac{\,\partial}{\,\partial x}\phi_{t_{0}}(t,\cdot)\Big|\right\|_{\BMO(\mathbb R)} \left[\frac{\,\partial}{\,\partial x}\tilde\phi_{t_0}(s,\cdot)\right]_{A_\infty(\mathbb R)} +\left\|\log\Big|\frac{\,\partial}{\,\partial x}\phi_{s}(t_0,\cdot)\Big|\right\|_{\BMO(\mathbb R)}\nonumber\\
&\ \ <\frac{\epsilon_0C_3(1+C_4 2^{-1}\epsilon_0)}{2C_3(1+C_4 2^{-1}\epsilon_0)}+\frac{\epsilon_0}{2}\nonumber\\
&\ \ =\epsilon_0.\nonumber
\end{align*}
This derives
\begin{align*}
\sup_{s\le t\le t_1}\left\{\left\|\log\Big|\dfrac{\partial}{\partial x}\phi_s(t,\cdot)\Big|\right\|_{\BMO(\mathbb R)},\,
\left\|\log\Big|\dfrac{\partial}{\partial x}\tilde\phi_t(s,\cdot)\Big|\right\|_{\BMO(\mathbb R)}\right\}<\epsilon_0.
\end{align*}
Using this estimate in {\bf Step 2}, we further have the following estimate
\begin{align*}
&\sup_{s\le t\le t_1}\left\{\left\|\log\Big|\dfrac{\partial}{\partial x}\phi_s(t,\cdot)\Big|\right\|_{\BMO(\mathbb R)},\,
\left\|\log\Big|\dfrac{\partial}{\partial x}\tilde\phi_t(s,\cdot)\Big|\right\|_{\BMO(\mathbb R)}\right\}\nonumber \\
&\quad\le \int_s^{t_1} C_3\left\|\dfrac{\partial}{\partial x}b(s,\cdot)\right\|_{\BMO(\mathbb R)} \exp\left(C_3C_4\int_s^{t_1}\left\|\dfrac{\partial}{\partial x}b(s,\cdot)\right\|_{\BMO(\mathbb R)}\,ds\right)\,ds\\
&\quad<2^{-1}\epsilon_0,\nonumber
\end{align*}
which implies that \eqref{est-log-bmo4} holds for all $t\in (s,t_1]$.

Since in \eqref{extend-time1} and \eqref{extend-time2} the extension of time only depends on $b$ itself,
we may iterate this argument finite times and conclude that \eqref{est-log-bmo4} holds for all $t\in (s,T_1]$.

\noindent{\bf Step 4 - completing argument}.
Since $b$ satisfies
$$\frac{\,\partial b(t,\cdot)}{\,\partial x}\in L^1(0,T;L^\infty(\R)),$$
we may choose a sequence of increasing numbers $\{T_i\}_{i=1,\cdots,k_0}$ such that
$T_1=0$, $T_{k_0}=T$ and
\begin{align*}
\frac{\int_{T_i}^{T_{i+1}} C_3\left\|\dfrac{\partial}{\partial x}b(s,\cdot)\right\|_{\BMO(\mathbb R)}\,ds}{ \exp\left(-\int_{T_i}^{T_{i+1}} C_3C_4\left\|\dfrac{\partial}{\partial x}b(s,\cdot)\right\|_{\BMO(\mathbb R)}\,ds \right)}=2^{-1}\epsilon_0
\ \ \forall\ \ i\in\{1,...,k_0-2\},
\end{align*}
and
\begin{align*}
\frac{\int_{T_{k_0}-1}^{T_{k_0}} C_3\left\|\dfrac{\partial}{\partial x}b(s,\cdot)\right\|_{\BMO(\mathbb R)}\,ds}{ \exp\left(-\int_{T_{k_0}-1}^{T_{k_0}}  C_3C_4\left\|\dfrac{\partial}{\partial x}b(s,\cdot)\right\|_{\BMO(\mathbb R)}\,ds \right)}\le 2^{-1}\epsilon_0
\end{align*}
If $t\in (T_1,T_2]$, then {\bf Step 3} gives
\begin{equation}\label{pri-bmo-1}
\left\|\log\Big|\frac{\,\partial}{\,\partial x}\phi(t,\cdot)\Big|\right\|_{\BMO(\mathbb R)}\le\int_0^t\frac{ C_3\left\|\dfrac{\partial}{\partial x}b(r,\cdot)\right\|_{\BMO(\mathbb R)}}{ \exp\left(-C_3C_4\int_0^{t}\left\|\dfrac{\partial}{\partial x}b(r,\cdot)\right\|_{\BMO(\mathbb R)}\,dr\right)}\,ds.
\end{equation}

Suppose that $t$ belongs to
$$
\text{some}\ \ (T_i,T_{i+1}]\ \ \text{with}\ \ 2\le i\le k_0-1.
$$
By using the semigroup property of the flow $\phi$, we have
$$\phi(t,x)=\phi_{T_i}(t,\cdot)\circ\phi_{T_{i-1}}(T_i,\cdot)\circ\cdots\phi_{T_1}(T_2,x).$$
By using Lemma \ref{bmo-comp}, Lemma \ref{weight-bmo} and {\bf Step 3}, we conclude
\begin{align*}
&\left\|\log\Big|\frac{\,\partial}{\,\partial x}\phi(t,\cdot)\Big|\right\|_{\BMO(\mathbb R)}\\
&\ \ =\left\|\log\Big|\frac{\,\partial}{\,\partial x}\phi_{T_{2}}(t, \phi_{T_1}(T_{2},\cdot))\Big|\right\|_{\BMO(\mathbb R)}\nonumber\\
&\ \ \le \left\|\log\Big|\frac{\,\partial}{\,\partial z}\phi_{T_{2}}(t,z)\large|_{z= \phi_{T_1}(T_2,\cdot)}\Big|\right\|_{\BMO(\mathbb R)}+\left\|\log\Big|\frac{\,\partial}{\,\partial x}\phi_{T_1}(T_2,\cdot)\Big|\right\|_{\BMO(\mathbb R)}\nonumber\\
&\ \ \le C_3\left\|\log\Big|\frac{\,\partial}{\,\partial x}\phi_{T_{2}}(t,\cdot)\Big|\right\|_{\BMO(\mathbb R)} \left[\Big|\frac{\,\partial}{\,\partial x}\tilde\phi_{T_2}(T_1,\cdot)\Big|\right]_{A_\infty(\mathbb R)} +\left\|\log\frac{\,\partial}{\,\partial x}\phi_{T_1}(T_2,\cdot)\right\|_{\BMO(\mathbb R)}\nonumber\\
&\ \ \le C_3(1+C_4 2^{-1}\epsilon_0)\left\|\log\Big|\frac{\,\partial}{\,\partial x}\phi_{T_{2}}(t,\cdot)\Big|\right\|_{\BMO(\mathbb R)} +2^{-1}\epsilon_0\nonumber\\
&\ \ \le C_3(1+C_4)\left\|\log\Big|\frac{\,\partial}{\,\partial x}\phi_{T_{2}}(t,\cdot)\Big|\right\|_{\BMO(\mathbb R)} +1\nonumber\\
&\ \ \le \big(C_3(1+C_4)\big)^2\left\|\log\Big|\frac{\,\partial}{\,\partial x}\phi_{T_{3}}(t,\cdot)\Big|\right\|_{\BMO(\mathbb R)}+ C_3(1+C_4)+1\nonumber\\
&\ \ \le \cdots\nonumber\\
&\ \ \le \big(C_3(1+C_4)\big)^{i-1}\left\|\log\Big|\frac{\,\partial}{\,\partial x}\phi_{T_{i}}(t,\cdot)\Big|\right\|_{\BMO(\mathbb R)}+ \sum_{j=0}^{i-2}\big(C_3(1+C_4)\big)^{j}\nonumber\\
&\ \ \le \big(C_3(1+C_4)+1\big)^{i}.\nonumber
\end{align*}
Let $\delta_0>0$ obey $$C_3\delta_0 e^{C_3C_4\delta_0}=2^{-1}\epsilon_0.$$ As
$$
\begin{cases}
\epsilon_0\le 1;\\
\delta_0<1;\\
t\in (T_{i},T_{i+1}],
\end{cases}
$$
by our choice of $\{T_i\}$ we find
$$(i-1)\delta_0<\int_{0}^t\left\|\dfrac{\partial}{\partial x}b(s,\cdot)\right\|_{\BMO(\mathbb R)}\,ds\le i\delta_0,$$
whence
\begin{align*}
\left\|\log\Big|\frac{\,\partial}{\,\partial x}\phi(t,\cdot)\Big|\right\|_{\BMO(\mathbb R)}\le \frac{\frac{1}{\delta_0}\int_{0}^t\left\|\dfrac{\partial}{\partial x}b(s,\cdot)\right\|_{\BMO(\mathbb R)}\,ds}{
\exp\left(-C\int_{0}^t\left\|\dfrac{\partial}{\partial x}b(s,\cdot)\right\|_{\BMO(\mathbb R)}\,ds  \right)}.
\end{align*}
This, together with \eqref{pri-bmo-1}, implies
\begin{equation*}
\left\|\log\Big|\frac{\,\partial}{\,\partial x}\phi(t,\cdot)\Big|\right\|_{\BMO(\mathbb R)}\le \frac{\int_0^{t} \frac{C_3}{\delta_0}\left\|\dfrac{\partial}{\partial x}b(r,\cdot)\right\|_{\BMO(\mathbb R)}}{\exp\left(-C\int_0^{t}\left\|\dfrac{\partial}{\partial x}b(r,\cdot)\right\|_{\BMO(\mathbb R)}\,dr\right)\,ds},
\end{equation*}
as desired.

\end{proof}

Rather surprisingly, the hypothesis
$$
\int_0^T\Bigg\|\frac{\,\partial b(t,\cdot)}{\,\partial x}\Bigg\|_{L^\infty(\R)}\,dt<\infty
$$
in Theorem \ref{pri-flow} can be replaced by a weaker one
$$
\int_0^T\Bigg\|\frac{\,\partial b(t,\cdot)}{\,\partial x}\Bigg\|_{\ast}\,dt<\infty
$$
in the following assertion.

\begin{thm}\label{pri-flow-2}
Let $$b(t,x):\,[0,T]\times\rr\mapsto \rr\ \ \text{be in}\ \ L^1(0,T;C^1(\rr))\ \ \text{with}\ \
\int_0^T\Bigg\|\frac{\,\partial b(t,\cdot)}{\,\partial x}\Bigg\|_{\ast}\,dt<\infty.
$$
Then there exists a unique flow $\phi(t,x)$ satisfying
$$
\begin{cases}
\dfrac{\partial}{\partial t}\phi(t,x)=b(t,\phi(t,x))\ \ &\forall\ \ (t,x)\in [0,T]\times\R;\\
\phi_0(x)=x\ \ &\forall\ \ x\in\R.
\end{cases}
$$
Moreover
$$
\left\| \log\Big|\dfrac{\partial}{\partial x}\phi(t,x)\Big|\right\|_{\BMO(\mathbb R)}
\le \frac{\int_{0}^t 2C_6\left\|\dfrac{\partial}{\partial x}b(s,)\right\|_{\BMO(\mathbb R)} \,ds}{ \exp\left(-2C_7\int_{0}^t\left\|\dfrac{\partial}{\partial x}b(s,)\right\|_{\BMO(\mathbb R)}\,ds\right)}\ \ \forall\ \ t\in [0,T].
$$
\end{thm}
\begin{proof} The existence and uniqueness has essentially been established in \cite{R2}. So it remains to verify the last $ \BMO(\R)$-size estimate.

For each $(k,t)\in\cn\times [0,T]$ set
$$\begin{cases}
v_k(t,x)=\min\left\{\max\{-k,\partial_x b(t,x)\},\,k\right\};\\
b_k(t,x)=b(t,0)+\int_0^x v_k(t,y)\,dy.
\end{cases}$$
Then
\begin{equation}\label{truncation}
\begin{cases}
\partial_x b_k(t,\cdot)\in L^1(0,T;L^\infty(\R));\\
\|v_k(t,\cdot)\|_{ \BMO(\R)}\le 2\|\partial_xb(t,\cdot)\|_{ \BMO(\R)};\\
\|v_k(t,\cdot)\|_{\ast}\le 2\|\partial_xb(t,\cdot)\|_{\ast}.
\end{cases}
\end{equation}
In accordance with Propositions \ref{cont-est}-\ref{cont-est-2}, we see that
$\{b_k\}$ and $b$ satisfy the Zygmund condition with a uniform constant.

Let $\{\phi_k,\phi\}$ be the unique flow pair generalized by $\{b_k(t,x),b(t,x)\}$.
Then by \cite[Proposition 4]{R2}, we see that $\phi(t,\cdot)$ and $\phi_k(t,\cdot)$ are locally H\"older continuous on $\R$ for each $t\in [0,T]$.
Moreover for each compact set $K\subset \R$, both $\phi(t,\cdot)$ and $\phi_k(t,\cdot)$ are H\"older continuous on
$K$ for each $t\in [0,T]$ with the H\"older exponent and constant depending only on
$$\int_0^t\|\partial_x b(s,\cdot)\|_\ast\,ds.$$

On the other hand, by the construction of $b_k$ and Proposition \ref{cont-est-2} we have
\begin{eqnarray*}
|b_k(t,x)-b(t,0)|\le C_5\|v_k(t,\cdot)\|_\ast|x|(1+|\log|x||)\le 2C_5\|\partial_xb(t,\cdot)\|_\ast|x|(1+|\log|x||),
\end{eqnarray*}
thereby getting that
$$\big\{|\phi_k(t,x)|:\,(t,x)\in \ [0,T]\times K\big\}$$
is uniformly bounded. Denote by
$$C_8(K):=\sup\left\{|\phi_k(t,x)|+|\phi(t,x)|:\,(t,x,k)\in \ [0,T]\times K\times\cn\right\}. $$
Then it holds for each $x\in K$  and all $0\le s<t\le T$that
\begin{eqnarray*}
\left|\phi_k(t,x)-\phi_k(s,x)\right|&&\le \int_s^t |b_k(r,\phi_k(r,x))|\,dr\\
&&\le \int_s^t \left(|b(r,0)|+2C_5\|\partial_xb(r,\cdot)\|_\ast|C_8(K)|(1+|\log|C_8(K)||)\right)\,dr.
\end{eqnarray*}
This, together with the previous discussion on the H\"older continuity in the spatial direction, implies that
$\{\phi_k\}_k$ are equicontinuous on $[0,T]\times K$. Applying the Arzel\'a-Ascoli theorem,
we conclude that there is a subsequence of $\{\phi_k\}_k$, denoted by $\{\phi_{K,k}\}_k$, such that
$\phi_{K,k}$ converges uniformly on $[0,T]\times K$.

By construction we have
$$
b_k(t,x)\to b(t,x)\ \ \text{as}\ \ k\to \infty,
$$
thereby concluding that if $(t,x)\in [0,T]\times K$ then
\begin{eqnarray*}
\lim_{k\to\infty}\phi_{K,k}(t,x)&&=x+\lim_{k\to\infty}\int_0^tb_{K,k}(s,\phi_{K,k}(s,x))\,ds\\
&&=x+\lim_{k\to\infty}\int_0^t\int_0^{\phi_{K,k}(s,x)}[v_{K,k}(s,y)-\partial_xb(s,y)]\,dy\,ds+
\lim_{k\to\infty}\int_0^t b(s,\phi_{K,k}(s,x))\,ds.
\end{eqnarray*}
Since
$$|\phi_k(s,x)|\le C_8(K),
$$
one has
$$\left|\int_0^t\int_0^{\phi_{K,k}(s,x)}[v_{K,k}(s,y)-\partial_xb(s,y)]\,dy\,ds\right|\le \int_0^T\int_{-C_8(K)}^{C_8(K)}|\partial_xb(s,y)|\,dy\,ds<\infty ,$$
and hence the dominated convergence theorem and continuity of $b(t,\cdot)$ guarantee
\begin{eqnarray*}
\lim_{k\to\infty}\phi_{K,k}(t,x)&&=x+\int_0^t b(s,\lim_{k\to\infty}\phi_{K,k}(s,x))\,ds.
\end{eqnarray*}
By choosing a sequence of increasing compacts $K_j$ such that $\R=\cup_jK_j$ and passing
to further subsequences, we see that there is a subsequence of $\{\phi_{k}\}$, still denoted by $\{\phi_{K,k}\}$,
such that $\phi_{K,k}(t,x)$ converges on $[0,T]\times\R$, and uniformly on any compact subset $[0,T]\times \tilde K$, and consequently,
\begin{eqnarray*}
\lim_{k\to\infty}\phi_{K,k}(t,x)&&=x+\int_0^t b(s,\lim_{k\to\infty}\phi_{K,k}(s,x))\,ds,\ \forall\,(t,x)\in [0,T]\times \R.
\end{eqnarray*}
By the uniqueness, we see that
$$\phi(t,x)=\lim_{k\to\infty}\phi_{K,k}(t,x),\, \forall\,(t,x)\in [0,T]\times \R, $$
and the convergence is uniform on any compact set.

Since $$b(t,x)\in L^1(0,T;C^1(\R)),$$
and so is any $b_k(t,x)$. Accordingly, the proof of Theorem \ref{pri-flow} yields that if $(t,x)\in [0,T]\times \R$ then
\begin{align*}
\log\Big|\dfrac{\partial}{\partial x}\phi(t,x)\Big|&=\int_0^t\dfrac{\partial}{\partial x}b(s,\phi(s, x))\,ds\\
&=\int_0^t \lim_{k\to\infty}v_k(s,\phi_k(s,x))\,ds\\
&=\lim_{k\to \infty} \log\Big|\dfrac{\partial}{\partial x}\phi_k(t,x)\Big|.
\end{align*}
By \eqref{truncation} and Theorem \ref{pri-flow}, we see that for each $k\in\cn$, it holds
$$
\left\| \log\Big|\dfrac{\partial}{\partial x}\phi_k(t,x)\Big|\right\|_{\BMO(\mathbb R)}
\le \frac{\int_{0}^t 2C_6\left\|\dfrac{\partial}{\partial x}b(s,)\right\|_{\BMO(\mathbb R)} \,ds}{ \exp\left(-2C_7\int_{0}^t\left\|\dfrac{\partial}{\partial x}b(s,)\right\|_{\BMO(\mathbb R)}\,ds\right)}.
$$
By this, the weak-$\ast$ compactness in $\BMO(\R)$, and the pointwise convergence of
$$\dfrac{\partial}{\partial x}\phi_k(t,x),$$ we conclude that the last estimation holds also for
$$\log\Big|\dfrac{\partial}{\partial x}\phi(t,x)\Big|,
$$
thereby completing the proof.
\end{proof}

\section{\large Proof of main results}\label{s4}

\begin{proof}[Proof of Theorem \ref{main}] The argument consists of three steps.
	
\medskip

\noindent{\bf Step 1 - an Orlicz space estimate}. Let $\mu$ denote the Gaussian measure on $\R$, i.e.,
$$\mu(x)=\frac{1}{\sqrt{2\pi}}\exp\left(-\frac{|x|^2}{2}\right),$$
and $\mathrm{div}_\mu b$ denotes the distributional divergence of $b$ with respect to $\mu$.
We say that a measurable function $$f\in \mathrm{Exp}_\mu(\frac{L}{\log L})$$ provided
$$\|f\|_{\mathrm{Exp}_\mu(\frac{L}{\log L})}=\inf\left\{\lambda>0:\int_{\R}\left[\exp\left(\frac{|f(x)|/\lambda}{1+\log^+(|f(x)|/\lambda)}\right)-1\right]\,d\mu\leq 1\right\}.$$

Let $b(t,x)$ obey \eqref{e11}. Then
\begin{equation}\label{growth-est}
\begin{cases}\frac{b(t,x)}{1+|x|\log^+|x|}\in L^1(0,T;L^\infty(\R));\\
\mathrm{div}_\mu b(t,x)\in L^1(0,T;\mathrm{Exp}_\mu(\frac{L}{\log L})).
\end{cases}
\end{equation}
As a matter of fact, the first estimate of \eqref{growth-est} follows from Proposition \ref{cont-est} as
$$\frac{|b(t,x)|}{1+|x|\log^+|x|}\le \frac{|b(t,x)-b(t,0) +b(t,0)|}{1+|x|\log^+|x|}\le |b(t,0)|+C\left\|\frac{\,\partial}{\,\partial x}b(t,\cdot)\right\|_{\ast}.$$
To verify the second relation in \eqref{growth-est},
set $$\beta(t)=|b(t,0)|+C\left\|\frac{\,\partial}{\,\partial x}b(t,\cdot)\right\|_{ \BMO(\R)}.
$$
Noting that
\begin{align*}
\int_\R\exp\bigg(\frac{c\big|xb(t,x)\big|}{1+\log^+(c\big|xb(t,x)\big|)}\bigg)\,d\mu(x)\le \int_\R\exp\big(\frac{c|x|(1+|x|\log^+|x|)\beta(t)}{1+\log^+(c|x|(1+|x|\log^+|x|)\beta(t))}\big)\,d\mu(x),
\end{align*}
we obtain
$$\big\| xb(t,x)\big\|_{\mathrm{Exp}_\mu(\frac{L}{\log L})}\le C\beta(t).$$
On the other hand, for a $ \BMO(\R)$-function $f$,
we utilize
the John-Nirenberg inequality:
$$|\{x\in I:\,|f(x)-f_I|>\lambda\}|\le c_1|I|\exp\left(-\frac{c_2\lambda}{\|f\|_{ \BMO(\R)}}\right)\ \ \forall\ \ \text{interval}\ I\subset\R$$
to obtain that if
$$
\begin{cases}
I=[x-r,x+1];\\
(x,r)\in\R\times[1,\infty);\\
\gamma(t)=\left\|\frac{\,\partial}{\,\partial x}b(t,\cdot)\right\|_{\ast};\\
\alpha=c_2\big(2\gamma(t)\big)^{-1},
\end{cases}
$$
then
$$|f_I|\le |f_{I}-f_{[-1,1]}|+|f_{[-1,1]}|\le C(1+\log^+|x|)\|f\|_{\ast},
$$
and hence
\begin{align*}
&\int_\R\exp\left(\alpha\left|\frac{\,\partial}{\,\partial x}b(t,x)\right|\right)\,d\mu(x)\\
&\quad\le
\int_{[-1,1]}\exp\Big(\alpha\left|\frac{\,\partial}{\,\partial x}b(t,x)\Big|\right)\,d\mu+ \sum_{k=1}^\infty\Bigg( \int_{[2^{k-1},2^k]}+\int_{[-2^k,-2^{k-1}]}\Bigg)\exp\left(\alpha\left|\frac{\,\partial}{\,\partial x}b(t,x)\right|\right)\,d\mu(x)\\
&\quad\le e^{2\alpha\gamma(t)} \sum_{k=0}^\infty \alpha 2^k e^{-2^{2k-1}+ck} \Bigg(\frac{\alpha\gamma(t)}{c_2-\alpha\gamma(t)}\Bigg)\le C.
\end{align*}
Consequently we achieve the desired inequality
$$\left\|\frac{\,\partial}{\,\partial x}b(t,\cdot)\right\|_{\mathrm{Exp}_\mu(\frac{L}{\log L})}\le \left\|\frac{\,\partial}{\,\partial x}b(t,\cdot)\right\|_{\mathrm{Exp}_\mu(L)}\le C\left\|\frac{\,\partial}{\,\partial x}b(t,\cdot)\right\|_{\ast}.
$$

\medskip
\noindent{\bf Step 2 - existence-uniqueness-size of flow}. Under \eqref{e11}
we conclude via Proposition \ref{cont-est}
for a.e. $t$, that $b$ is in the Zygmund class, which implies that the flow exists and is unique;
see \cite{R2} for instance.

Moreover, from {\bf Step 1} above it follows that $b$ satisfies requirements from \cite[Main Theorem]{CJMO1} and so that $\phi(t, x)$ is absolutely continuous and differentiable.
Indeed, by using \cite[Theorem 1.2]{CJMO1} and that $b(t,\cdot)$ is in the Zygmund class,
one can deduce that
$$\left|\frac{\partial}{\partial x}\phi(t, \cdot)\right|\left(1+\log^+\left|\frac{\partial}{\partial x}\phi(t, \cdot)\right|\right)^q\in L^1_\loc(\R)$$
for any $q\in [1,\infty)$. As $\partial_x b(t,x)\in \BMO(\R)$ is locally exponentially
integrable, we deduce that
\begin{equation*}\label{chain-rule}
\dfrac{\partial}{\partial t}\left(\dfrac{\partial}{\partial x}\, \phi(t, x)\right)=\left(\frac{\,\partial}{\,\partial z}b(t,z)|_{z=\phi(s,x)}\right)\dfrac{\partial}{\partial x}\, \phi(t, x)
\end{equation*}
and
\begin{equation}\label{formula-density}
\log\left|\dfrac{\partial}{\partial x}\, \phi(t, x)\right|=\int_0^t\frac{\,\partial}{\,\partial x}b(s,\phi(s,x))\,ds.
\end{equation}
For $\epsilon>0$ and $x\in\R$ set
$$
\begin{cases}
0\le \rho\in C^\infty_c(\R);\\
\supp \rho\subset (-1,1);\\
\int_\R \rho(x)\,dx=1;\\
\rho_\epsilon(x)=\frac{1}{\epsilon}\rho\left(\frac x{\epsilon}\right);\\
b_\epsilon(t,x)=b(t,\cdot)\ast\rho_\epsilon(x).
\end{cases}
$$
Note that $$
\frac{\partial}{\partial x}b(t,x)\in L^1(0,T; \BMO(\R))\Rightarrow
\frac{\partial}{\partial x}b_\epsilon(t,x)\in L^1(0,T; \BMO(\R))\cap L^1(0,T;C^\infty(\R)).
$$
Thus we have
$$
\int_0^t\Bigg\|\frac{\partial}{\partial x}b_\epsilon(s,\cdot)\Bigg\|_{\BMO(\mathbb R)}\,ds\le \int_0^t\Bigg\|\frac{\partial}{\partial x}b(s,\cdot)\Bigg\|_{\BMO(\mathbb R)}\,ds\ \ \forall\ \ t\in (0,T]
$$
and so for any $\epsilon\in (0,1)$
$$\Bigg\|\frac{\partial}{\partial x}b_\epsilon(t,\cdot)\Bigg\|_{\ast}
\le 2\Bigg\|\frac{\partial}{\partial x}b(t,\cdot)\Bigg\|_{\ast}\ \ \text{for a.e.}\ \ t\in (0,T].
$$

Let $\phi_\epsilon(t,x)$ be the flow generated by $b_\epsilon$, i.e.,
$$\frac{\,\partial}{\,\partial t}\phi_\epsilon(t,x)=b_\epsilon(t,\phi_\epsilon(t,x)).$$
Then Theorem \ref{pri-flow-2} is utilized to imply
\begin{align*}
\left\| \log\Big|\dfrac{\partial}{\partial x}\phi_\epsilon(t,\cdot)\Big|\right\|_{\BMO(\mathbb R)}
&\le \frac{\int_{0}^t 2C_6\left\|\dfrac{\partial}{\partial x}b_\epsilon(s,\cdot)\right\|_{\BMO(\mathbb R)} \,ds}{ \exp\left(-2C_7\int_{0}^t\left\|\dfrac{\partial}{\partial x}b_\epsilon(s,\cdot)\right\|_{\BMO(\mathbb R)}\,ds\right)}\\
&\le \frac{\int_{0}^t 2C_6\left\|\dfrac{\partial}{\partial x}b(s,\cdot)\right\|_{\BMO(\mathbb R)} \,ds}{ \exp\left(-2C_7\int_{0}^t\left\|\dfrac{\partial}{\partial x}b(s,\cdot)\right\|_{\BMO(\mathbb R)}\,ds\right)}\ \ \ \ \forall\ \ \ \ \epsilon>0.
\end{align*}
The proof of \cite[Main Theorem]{CJMO1} infers that, up to a subsequence $\{\epsilon_k\}_{k\in\cn}$,
$$\lim_{k\to\infty}\phi_{\epsilon_k}(t,x)=\phi(t,x)\ \ \forall\ \ t\in (0,T].
$$
From this, \eqref{formula-density}  and the weak-$^\ast$ compactness in $ \BMO(\R)$,  we conclude that
$\frac{\,\partial}{\,\partial x}\phi$ is the weak-$^\ast$ limit of $\frac{\,\partial}{\,\partial x}\phi_{\epsilon_k}$
for each $t\in (0,T]$. This implies
\begin{align*}
\left\| \log\Big|\dfrac{\partial}{\partial x}\phi(t,\cdot)\Big|\right\|_{\BMO(\mathbb R)}
&\le \frac{\int_{0}^t 2C_6\left\|\dfrac{\partial}{\partial x}b(s,\cdot)\right\|_{\BMO(\mathbb R)} \,ds}{ \exp\left(-2C_7\int_{0}^t\left\|\dfrac{\partial}{\partial x}b(s,\cdot)\right\|_{\BMO(\mathbb R)}\,ds\right)},
\end{align*}
namely, the size estimate \eqref{bmo-flow-1} holds.

\medskip

\noindent{\bf Step 3 - $A_\infty(\R)$ density of flow}.	It remains to show that for each $t\in [0,T]$, $$\left|\dfrac{\partial}{\partial x}\phi(t,\cdot)\right|$$ is an $A_\infty(\mathbb R)$-weight. But, from Theorem \ref{bmo-transport} (to be proved later on), we see that
$$
u_0\in \BMO(\R)\Rightarrow u_0\circ\phi(t,\cdot)\in  \BMO(\R)\ \ \forall\ \ t\in (0,T].
$$
Then we apply \cite[Theorem]{J} to conclude that
	for each $t\in [0,T]$, $$\left|\frac{\,\partial}{\,\partial x}\phi(t,x)\right|$$ is an $A_\infty(\mathbb R)$-weight.

\end{proof}

\begin{proof}[Proof of Theorem \ref{bmo-transport}] The argument consists of three steps.
	
	\medskip
	\noindent{\bf Step 1 - existence of solution}. Let $\phi$ be the flow generated by $b$, i.e.,
$$
\begin{cases}
\dfrac{\partial}{\partial t}\, \phi(t,x)=b(t,\phi(t,x))\ \ &\forall\ \ (t,x)\in (0,T]\times\R;\\
\phi_0(x)=x\ \ &\forall\ \ x\in\R.
\end{cases}
$$
Then the same proof of \cite[Theorem 1]{CJMO0} derives that
$u_0\circ \phi$ is a solution to the transport equation.

\medskip
\noindent{\bf Step 2 - size of solution}. Let $\epsilon_0$ be the same as in Lemma \ref{weight-bmo}, and
$$
\delta_0>0\ \ \&\ \ 2C_6\delta_0 e^{2C_7\delta_0}=2^{-1}\epsilon_0.
$$
We choose a sequence of increasing numbers
$$
0=T_0<T_1<\cdots<T_{k_0}=T
$$
such that
$$
\frac{\int_{T_{i-1}}^{T_{i}} 2C_6\left\|\dfrac{\partial}{\partial x}b(s,\cdot)\right\|_{\BMO(\mathbb R)}\,ds}{ \exp\left(-\int_{T_{i-1}}^{T_{i}} 2C_7\left\|\dfrac{\partial}{\partial x}b(s,\cdot)\right\|_{\BMO(\mathbb R)}\,ds \right)}=2^{-1}\epsilon_0\ \ \ \forall\ \ \ i\in\{1,...,k_0-1\},
$$
and
$$
\frac{\int_{T_{k_0}-1}^{T_{k_0}} 2C_6\left\|\dfrac{\partial}{\partial x}b(s,\cdot)\right\|_{\BMO(\mathbb R)}\,ds}{ \exp\left(-\int_{T_{k_0}-1}^{T_{k_0}}  2C_7\left\|\dfrac{\partial}{\partial x}b(s,\cdot)\right\|_{\BMO(\mathbb R)}\,ds \right)}\le 2^{-1}\epsilon_0.
$$
Suppose that $t$ belongs to
$$
\text{some}\ \ (T_i,T_{i+1}]\ \ \text{where}\ \ i=0,...,k_0-1.
$$

If $i=0$, then by Lemma \ref{bmo-comp} and Lemma \ref{weight-bmo}, we obtain
\begin{align}\label{bmo-solution}
\left\|u(t,\cdot)\right\|_{\BMO(\mathbb R)}&\le C_3\|u_0\|_{\BMO(\mathbb R)} \left(1+C_4\left\| \log\Big|\dfrac{\partial}{\partial x}\tilde\phi_{t}(0,\cdot)\Big|\right\|_{\BMO(\mathbb R)}\right)\nonumber\\
&\le C_3\|u_0\|_{\BMO(\mathbb R)} \left(1+\frac{2C_4C_6 \int_0^{t}\left\|\dfrac{\partial}{\partial x}b(s,\cdot)\right\|_{\BMO(\mathbb R)}\,ds}{\exp\left(-\int_{0}^{t} 2C_7\left\|\dfrac{\partial}{\partial x}b(s,\cdot)\right\|_{\BMO(\mathbb R)}\,ds \right)}\right)\\
&\le C_3\|u_0\|_{\BMO(\mathbb R)}\exp\left(\int_{0}^{t} C\left\|\dfrac{\partial}{\partial x}b(s,\cdot)\right\|_{\BMO(\mathbb R)}\,ds \right).\nonumber
\end{align}

Suppose next $i\ge 1$. By the semigroup property of the flow, we may write
$$u(t,x)=u_0\circ\phi_{T_i}(t,\cdot)\circ \cdots\circ \phi_{T_{0}}(T_1,x).$$
By Theorem \ref{main}, we have
\begin{equation}\label{bmo-growth-2}
\left\| \log\Big|\dfrac{\partial}{\partial x}\tilde\phi_{t}(T_i,\cdot)\Big|\right\|_{\BMO(\mathbb R)}
\le \frac{\int_{T_i}^t 2C_6\left\|\dfrac{\partial}{\partial x}b(s,\cdot)\right\|_{\BMO(\mathbb R)} \,ds}{ \exp\left(-2C_7\int_{T_i}^t\left\|\dfrac{\partial}{\partial x}b(s,\cdot)\right\|_{\BMO(\mathbb R)}\,ds\right)}
\le 2^{-1}\epsilon_0\ \ \forall\ \ t\in(T_i,T_{i+1}].
\end{equation}
A combination of \eqref{bmo-growth-2} and Lemma \ref{weight-bmo} derives
$$
\begin{cases}
\left|\dfrac{\partial}{\partial x}\tilde\phi_{t}(T_i,\cdot)\right|\in A_\infty(\mathbb R);\\
\left[\Big|\dfrac{\partial}{\partial x}\tilde\phi_{t}(T_i,\cdot)\Big|\right]_{ A_\infty(\mathbb R)}\le 1+C_4\left\| \log\Big|\dfrac{\partial}{\partial x}\tilde\phi_{t}(T_i,\cdot)\Big|\right\|_{\BMO(\mathbb R)}.
\end{cases}
$$
Then Lemma \ref{bmo-comp} implies
$$\left\|v\circ \phi_{T_i}(t,\cdot)\right\|_{\BMO(\mathbb R)}\le C_3\|v\|_{\BMO(\mathbb R)}\left(1+C_4\left\| \log\Big|\dfrac{\partial}{\partial x}\tilde\phi_{T_i}(t,\cdot)\Big|\right\|_{\BMO(\mathbb R)}\right)\ \ \forall\ \ v\in  \BMO(\R).$$
Upon repeating this argument for $i$ times more, we gain
\begin{align*}
\left\|u(t,\cdot)\right\|_{\BMO(\mathbb R)}
&=\left\|u_0\circ\phi_{T_i}(t,\cdot)\circ \cdots\circ \phi_{T_{0}}(T_1,\cdot)\right\|_{\BMO(\mathbb R)}\\
&\le C_3^{i+1}\|u_0\|_{\BMO(\mathbb R)} \frac{\prod_{j=1}^{i} \left(1+C_4\left\| \log\Big|\dfrac{\partial}{\partial x}\tilde\phi_{T_j}(T_{j-1},\cdot)\Big|\right\|_{\BMO(\mathbb R)}\right)}{
\left(1+C_4\left\| \log\Big|\dfrac{\partial}{\partial x}\tilde\phi_{T_i}(t,\cdot)\Big|\right\|_{\BMO(\mathbb R)}\right)^{-1}}\\
&\le  C_3^{i+1}\left(1+C_42^{-1}\epsilon_0\right)^{i+1}\|u_0\|_{\BMO(\mathbb R)} \\
&\le \|u_0\|_{\BMO(\mathbb R)} \exp\left(C\int_{0}^t\left\|\dfrac{\partial}{\partial x}b(s,\cdot)\right\|_{\BMO(\mathbb R)}\,ds\right),
\end{align*}
where in the last inequality we have used
$$i\delta_0<\int_{0}^t\left\|\dfrac{\partial}{\partial x}b(s,\cdot)\right\|_{\BMO(\mathbb R)}\,ds\le (i+1)\delta_0.$$
This, together with \eqref{bmo-solution}, gives the desired size estimate.

\medskip
\noindent{\bf Step 3 - uniqueness of solution}.  This follows easily as an application of the renormalized property
of solutions established by DiPerna-Lions \cite{DPL} and the well-posedness of solutions
in $L^\infty(0,T;L^\infty(\R))$ established in \cite{CJMO};
see the proof of \cite[Thoerem 1]{CJMO0} for instance.
\end{proof}

{\it Acknowledgement}.
This work may be treated as a continuation of R. Jiang's project joint with
A. Clop, J. Mateu and J. Orobitg at Department of Mathematics, Autonomous University of Barcelona and R. Jiang would like to thank the department for its warm hospitality. In addition, R. Jiang was supported in part by National Natural Science Foundation
of China (11671039 \& 11771043); K.W. Li was supported by Juan de la Cierva - Formaci\'on 2015 FJCI-2015-24547, by the Basque Government through the BERC
2018-2021 program and by Spanish Ministry of Economy and Competitiveness
MINECO through BCAM Severo Ochoa excellence accreditation SEV-2013-0323
and through project MTM2017-82160-C2-1-P funded by (AEI/FEDER, UE) and
acronym ``HAQMEC''; J. Xiao was supported by NSERC of Canada (\#20171864).

\

{\it  Conflict of Interest Statement}.
The authors declare that there is no conflict of interest regarding the publication of this paper.

\noindent Renjin Jiang \\
\noindent  Center for Applied Mathematics, Tianjin University, Tianjin 300072, China\\
\noindent {rejiang@tju.edu.cn}

\

\noindent Kangwei Li\\
\noindent Basque Center for Applied Mathematics, Mazarredo, 14. 48009 Bilbao
Basque Country, Spain\\
\noindent {kli@bcamath.org}

\

\noindent  {Jie Xiao} \\
\noindent {Department of Mathematics and Statistics, Memorial University, St. John's, NL A1C 5S7, Canada}\\
\noindent {jxiao@mun.ca}

\end{document}